\documentclass[a4paper,reqno,10pt]{amsart} 
\usepackage{amsmath,amssymb,amsfonts,graphicx,epsfig,color,url}
\usepackage{latexsym,mathrsfs,cite}

\usepackage{float}
\usepackage{mathtools}
\usepackage{hyperref} 
\usepackage{comment}
\usepackage[title,titletoc,toc]{appendix}

\usepackage{graphicx}
\usepackage{subcaption}

\numberwithin{equation}{section}

\title[Traveling waves for QHD with nonlinear viscosity]{Traveling waves for  quantum hydrodynamics with nonlinear viscosity}

\author{Corrado Lattanzio and Delyan Zhelyazov}
\address[Corrado Lattanzio]{DISIM, Department of Information Engineering, Computer Science and Mathematics \\ University of L'Aquila, Italy}
\email{corrado@univaq.it}
\address[Delyan Zhelyazov]{DISIM, Department of Information Engineering, Computer Science and Mathematics \\ University of L'Aquila, Italy}\email{delyanatanasov.zhelyazov@univaq.it}

\newtheorem{lemma}{Lemma}
\newtheorem{corollary}[lemma]{Corollary}

\theoremstyle{remark}
\newtheorem{remark}[lemma]{Remark}

\begin{document}

\begin{abstract}
In this paper we study existence of traveling waves  for   1-D compressible Euler system  with dispersion (which models quantum effects  through the Bohm potential) and nonlinear viscosity in the context of quantum hydrodynamic models for superfluidity. The existence of 
profiles is proved for appropriate (super-- or sub-- sonic)  end states defining Lax shocks  for the underlying Euler system formulated in terms of density and velocity without restrictions for the viscosity and dispersion parameters. On the other hand, 
the interplay of the dispersion  and the  viscosity plays a crucial role in proving 
the existence of  oscillatory profiles, showing in this way  how the dispersion plays a significant role in certain regimes.
Numerical experiments are also provided to analyze the sensitivity of such profiles with respect to the viscosity/dispersion terms and with respect to the 
nearness to vacuum.
\end{abstract}


\maketitle

\section{Introduction}\label{sec:intr}
The aim of this paper is the study of traveling waves (referred to as \emph{dispersive shocks}) for the following one dimensional
quantum hydrodynamics (QHD) system with nonlinear viscosity:
\begin{equation}
\label{eq_sys_n_intro}
\begin {cases}
\displaystyle{\rho_t+m_x=0,} & \\
\displaystyle{m_t+\left(\frac{m^2}{\rho}+p(\rho) \right)_x=\epsilon \mu \rho \left(\frac{m_{x}}{\rho}\right)_x+\epsilon^2 k^2 \rho \left(\frac{(\sqrt{\rho})_{xx}}{\sqrt{\rho}}\right)_x,} & \\
\end{cases}
\end{equation}
where $\rho\geq0$ is the density, $m=\rho u$ is the momentum, where $u$ stands fo the velocity, and $p(\rho)=\rho^\gamma$ for $\gamma \geq 1$ is the pressure. The positive constants $0<\epsilon\ll1$, $\mu$, and $k>0$ defines the viscosity ($\epsilon \mu$) and the dispersive ($\epsilon^2 k^2$) coefficients.
The shape of the dispersive term is known as the Bohm potential, while, for the specific form of the nonlinear viscosity, we refer to the theory of superfluidity (for instance, see \cite[page 109]{Khalatnikov}), and in particular it describes the interaction of a super fluid and a normal fluid. Moreover, it can be also interpreted as a viscosity term  describing the  interaction of the fluid with a background.

The formulation of models involving dispersive temrs dates back to the 60s \cite{Sagdeev,Zak} and 70s \cite{Gurevich}; for more recent studies, see also \cite{Gurevich1,Nov,Hoefer}.
%
%
Moreover, 
the mathematical treatment of such systems, under different perspectives, and using different techniques,  can be found in \cite{AM1,AM2,AMtf,AMDCDS,AS,Michele1,Michele,DM,DM1,DFM,GLT,BGL-V19}.

Concerning in particular the study of dispersive shocks,
which is the topic of the present work, we refer to  \cite{Humpherys}, where the spectral theory of the linearized operator around the profiles has been discussed in the case of  $p$-system with real viscosity and linear capillarity. Moreover, the case of the Euler formulation \eqref{eq_sys_n_intro}, but with linear viscosity, is investigated in \cite{Zhelyazov}. In this last paper, the Authors in particular proved existence of traveling waves for arbitrary strong shocks of the underlying Euler system, under appropriate  conditions on the ratio $\mu/k$. 
For the nonlinear viscosity case inder examination here, 
since all along the profile we are interested in the density $\rho$ will stay away from zero,  the fluid velocity $u$ is  well defined and, dividing the second equation by $\rho$, we can rewrite \eqref{eq_sys_n_intro} using $(\rho, u)$ variables (see \eqref{eq_sys2_1}-\eqref{eq_sys2_2} below), to obtain a system in conservation form, as already done in \cite{Gasser}.
Then, using this formulation, we are able to prove existence of profiles without the aforementioned restrictions on the viscosity and dispersion coefficients.  However, as already noticed in \cite{Zhelyazov}, the smallness of the ratio  $\mu/k$ determines the oscillatory behavior of the profiles, as we checked also numerically.

The remaining part of the paper is organized as follows. In Section \ref{QHD_nonlinear_dissipation} we introduce the appropriate rewriting of the system in terms of density and velocity, and then we derive the equation solved by a profile. In Section \ref{sec:existence} we present the existence result for the profile, also introducing sufficient conditions  in terms of admissibility of end states. Finally,
 Section \ref{sec:sens} contains the numerical experiments describing  the sensitivity of the profiles, as we vary the viscosity/dispersion ratio, and as the density end state approaches vacuum.
\section{Preliminaries}
\label{QHD_nonlinear_dissipation}
In order to study traveling waves for the QHD model \eqref{eq_sys_n_intro}, in the following we shall first rewrite it in a conservative way using the variables $(\rho, u)$ to then obtain a suitable formulation to efficiently study the dynamical system governing the profile defining the traveling wave. 
This formulation is justified in our analysis also in view of the fact that the profile in the density stays away from vacuum and therefore the fluid velocity $u$ is clearly well defined. In addition to that, is it worth to recall here  the connection of the QHD system \eqref{eq_sys_n_intro} with the Schr\"odinger equation for the corresponding wavefunction, which is defined starting from the hydrodynamic variables $\rho$ and $u$ 
 after one introduce the phase $\phi$ as $u=\phi_x$; among others, see \cite{Gasser}. Hence, we define the enthalpy $h(\rho)$ by $h(\rho)=\ln\rho$ for $\gamma = 1$, and $h(\rho)=\frac{\gamma}{\gamma-1}\rho^{\gamma-1}$ for $\gamma>1$ to obtain
\begin{equation*}
h(\rho)_x=\frac{1}{\rho}(p(\rho))_x.\nonumber
\end{equation*}
After dividing the second equation of \eqref{eq_sys_n_intro} by $\rho$, we obtain
\begin{align}
&\rho_t+(\rho u)_x=0, \label{eq_sys1_1}\\
&\frac{(\rho u)_t}{\rho}+\frac{1}{\rho}\Big(\rho u^2 + p(\rho) \Big)_x=\epsilon \mu \Big(\frac{(\rho u)_{x}}{\rho}\Big)_x+\epsilon^2 k^2 \Big(\frac{(\sqrt{\rho})_{xx}}{\sqrt{\rho}}\Big{)}_x. \label{eq_sys1_2}
\end{align}
Then, we simplify the momentum equation by using the continuity equation
\begin{align*}
&\frac{(\rho u)_t}{\rho}+\frac{1}{\rho}(\rho u^2)_x=\frac{\rho_t u}{\rho}+u_t + \frac{1}{\rho}\rho_x u^2 + (u^2)_x\nonumber\\
&=-\frac{(\rho u)_x u}{\rho}+u_t + \frac{1}{\rho}\rho_x u^2 + (u^2)_x=u_t + \frac{(u^2)_x}{2}
\end{align*}
and write the system \eqref{eq_sys1_1}-\eqref{eq_sys1_2} in conservative form using the velocity and the enthalpy as follows:
\begin{align}
&\rho_t+(\rho u)_x=0,\label{eq_sys2_1}\\
&u_t + \frac{(u^2)_x}{2}+(h(\rho))_x=\epsilon \mu \Big(\frac{(\rho u)_{x}}{\rho}\Big)_x+\epsilon^2 k^2 \Big(\frac{(\sqrt{\rho})_{xx}}{\sqrt{\rho}}\Big{)}_x. \label{eq_sys2_2}
\end{align}
As said before, we are interested in  traveling wave profiles for \eqref{eq_sys2_1}-\eqref{eq_sys2_2}, namely solution of this system of the form
\begin{equation}\label{eq:prof}
\rho(t,x)=P\Big(\frac{x-s t}{\epsilon}\Big)^2,\mbox{ }u(t,x)=U\Big(\frac{x-st}{\epsilon}\Big),
\end{equation}
where $s \in \mathbb{R}$ is the speed of the traveling wave and the limiting  states
\begin{equation*}
P^{\pm}=\lim_{y\rightarrow \pm \infty}P(y),\mbox{ }u^{\pm}=\lim_{y\rightarrow \pm \infty}U(y),
\end{equation*}
for $\rho^\pm = (P^2)^\pm$, are assumed to satisfy the Rankine--Hugoniot conditions
for the underlying system
\begin{align}
&\rho_t+(\rho u)_x=0, \label{euler1} \\
&u_t + \frac{(u^2)_x}{2}+(h(\rho))_x=0,
\label{euler2}
\end{align}
namely
\begin{align}
s((P^2)^+-(P^2)^-) &=(P^2 u)^+ - (P^2 u)^-,\label{RHnon1}\\
s(u^+-u^-)&=\Big(\frac{u^2}{2}+h(P^2)\Big)^+-\Big(\frac{u^2}{2}+h(P^2)\Big)^-.
\label{RHnon2}
\end{align}
Plugging the ansatz \eqref{eq:prof} in  \eqref{eq_sys2_1}-\eqref{eq_sys2_2} we obtain
\begin{align}
&-s(P^2)'+(P^2 U)'=0,\label{profile1_1}\\
&-s U'+\frac{(U^2)'}{2}+(h(P^2))'=\mu \Big(\frac{(P^2 U)'}{P^2}\Big)'+k^2\Big(\frac{P''}{P}\Big)',\label{profile1_2}
\end{align}
where $P=P(y)$, $U=U(y)$ and $'$ denotes $d/dy$. To get a simpler system we substitute $(P^2 U)'$ from equation \eqref{profile1_1} in the first term of the right-hand side of \eqref{profile1_2} and the dynamical systems rewrites as follows
\begin{align}
&-s(P^2)'+(P^2 U)'=0,\label{profile2_1}\\
&-s U'+\frac{(U^2)'}{2}+(h(P^2))'=\mu \Big(\frac{s (P^2)'}{P^2}\Big)'+k^2\Big(\frac{P''}{P}\Big)'.
\label{profile2_2}
\end{align}
After integration of equation \eqref{profile2_1}  up to $\pm\infty$ we end up with
\begin{equation}
U=s-\frac{A}{P^2} \label{eq_v},
\end{equation}
where
\begin{equation}\label{eq:Aadd} 
A=(s-u^{\pm})(P^{\pm})^2.
\end{equation}
Similarly,  integration of equation \eqref{profile2_2} yields to
\begin{equation}
-s U +\frac{U^2}{2}+h(P^2)=2 s \mu \frac{P'}{P}+k^2\frac{P''}{P}-B \label{eq2D0},
\end{equation}
where
\begin{equation}\label{eq:Badd} 
B=s u^{\pm}-\frac{(u^{\pm})^2}{2}-h((P^{\pm})^2).
\end{equation}
In view of  \eqref{eq_v} we can eliminate the variable $U$ in \eqref{eq2D0} to obtain the second order equation
\begin{equation}
P''=\frac{f(P)}{k^2}-\frac{2 s \mu}{k^2} P',\label{eq_second_order}
\end{equation}
where
\begin{equation*}
f(P)=\Big(\frac{1}{2}\Big(\frac{A^2}{P^4}-s^2\Big)+h(P^2)+B\Big)P.
\end{equation*}
Finally, in view of \eqref{RHnon1}-\eqref{RHnon2}, the constants $A$ and $B$ in $f(P)$ can be expressed only n terms of $P^{\pm}$:
\begin{align*}
&f(P)\\
&=\Big(\frac{(P^+P^-)^4}{P^4}\frac{h((P^+)^2)-h((P^-)^2)}{(P^+)^4-(P^-)^4}+h(P^2)-\frac{(P^+)^4 h((P^+)^2)-(P^-)^4 h((P^-)^2)}{(P^+)^4-(P^-)^4}\Big)P.
\end{align*}
\section{Global existence of profiles}\label{sec:existence}
The second order equation \eqref{eq_second_order} is the starting point in the path of the   
 global existence result of traveling wave profiles for QHD with nonlinear viscosity  \eqref{eq_sys1_1}-\eqref{eq_sys1_2}. 
 For the latter, in contrast to the linear viscosity case investigated in \cite{Zhelyazov}, here  we do not have to impose conditions on the viscosity and dispersion coefficients in order to prove existence of profiles. However, as expected (see also the discussion in \cite{Gasser}),  the effect of the dispersion term is clearly visible in terms of the presence of oscillatory profiles in certain regimes; see the condition of Lemma \ref{global_existence_nonlinear_dissipation}  and the discussions of Section \ref{subsec:sensvisc} below.
 
To start with, we introduce the variable $P'=Q$ and rewrite \eqref{eq_second_order} as a first--order system as follows
\begin{align}
P'&=Q=f_1,\label{first_order_system_1}\\
Q'&=\frac{f(P)}{k^2}-\frac{2 s \mu}{k^2}Q=f_2,\label{first_order_system_2}
\end{align}
which has equilibria $(P^{\pm},0)$, being the end states $P^{\pm}$ positive roots of $f(P)=0$, as it is manifest by using relation  \eqref{eq:Aadd} and  \eqref{eq:Badd} 
in the expression defining $f$. 
Moreover, a direct calculation shows  
\begin{equation} \label{eq:fprime}
f'(P) = \begin{dcases}
-\frac{3 A^2 }{2 P^4}+\ln(P^2)+2+B-\frac{s^2}{2}, & \gamma = 1\\
- \frac{3 A^2 }{2 P^4}+\frac{\gamma(2 \gamma-1)}{\gamma-1} P^{2(\gamma-1)}+B-\frac{s^2}{2}, & \gamma >1;
\end{dcases}
\end{equation}
\begin{equation*}
f''(P) = \begin{dcases}
\frac{6A^2}{P^5} + \frac{2}{P}, & \gamma = 1\\
\frac{6A^2}{P^5} + 2\gamma(2\gamma-1)P^{2\gamma-3}, & \gamma >1.
\end{dcases}
\end{equation*}
In particular, for $P>0$ we have $f''(P)>0$ and therefore $P^\pm$ are the only two positive zeroes of $f$.

The proof of existence of an heteroclinic orbit  for   \eqref{first_order_system_1}-\eqref{first_order_system_2} between these equilibria
is obtained separately in the two cases $s>0$ and $s<0$, and under appropriate conditions for the end states $P^\pm$. 
The latter will be then interpreted afterwards  in terms of super-- and sub--sonic property for the corresponding end states $(\rho^\pm, u^\pm,s)$ defining a Lax shock 
for the $\epsilon =0$ reduced system \eqref{euler1}-\eqref{euler2}. 
The result will be obtained by showing the existence of  a Lyapunov function for that system  and then via an   application of  the LaSalle invariance principle.
For this, a crucial role will be played by the following  reduced system
\begin{align}
P'&=Q,\label{first_order_system_reduced_1}\\
Q'&=\frac{f(P)}{k^2}.\label{first_order_system_reduced_2}
\end{align}
It has (conserved) energy 
\begin{equation*}
H(P,Q)=F(P)-\frac{Q^2}{2}-F(P^-),
\end{equation*}
where 
\begin{align*}
F(P): = & \frac{1}{k^2}\int^P f(z) dz \\
= & 
\begin{dcases}
\frac{1}{k^2}\Big(-\frac{A^2}{4 P^2}+\frac{1}{2}\Big(B-\frac{s^2}{2}-1\Big)P^2+\frac{1}{2}P^2 \ln(P^2)\Big), & \gamma = 1\\
\frac{1}{4 k^2}\Big( -\frac{A^2}{P^2}+(2 B - s^2) P^2 +\frac{2 }{\gamma - 1}P^{2 \gamma} \Big), & \gamma >1.
\end{dcases}
\end{align*} 
In particular, we will show  that there exists a homoclinic loop for \eqref{first_order_system_reduced_1}-\eqref{first_order_system_reduced_2}, which confines the heteroclinic orbit we are looking for; see Figure \ref{fig_heteroclinic2}.
%
%
%
\begin{lemma}
\label{global_existence_nonlinear_dissipation}
Suppose that the end states $P^{\pm}$, $U^{\pm}$ and the speed $s$ satisfy the Rankine-Hugoniot conditions \eqref{RHnon1}-\eqref{RHnon2}.
\begin{enumerate}
\item[\textit{(i)}]
If $s>0$ and $0<P^+<P^-$, then there exists a heteroclinic, connecting $[P^-,0]$ to $[P^+,0]$. If in addition
\begin{equation*}
\frac{s \mu}{k}<\sqrt{-f'(P^+)},
\end{equation*}
then the heteroclinic is non-monotone.
\item[\textit{(ii)}]
If $s<0$ and $0<P^-<P^+$, then there exists a heteroclinic, connecting $[P^-,0]$ to $[P^+,0]$. If in addition
\begin{equation*}
-\frac{s \mu}{k}<\sqrt{-f'(P^-)},
\end{equation*}
then the heteroclinic is non-monotone.
\end{enumerate}
\end{lemma}
\begin{proof} 
\emph{Case (i).}
First of all, let us observe that $A\neq 0$. Indeed, if not, being $P^\pm \neq 0$, from \eqref{eq:Aadd} we readily obtain $u^+ = s = u^-$. Then, \eqref{RHnon2}  and the strict monotonicity of $h$ on $\rho >0$ imply $\rho^+ = (P^2)^+ = (P^2)^- = \rho^-$; a contradiction.

Now, from \eqref{eq:fprime} we know that $f'(P) \to -\infty$ as $P\to 0^+$ and  
$f'(P) \to +\infty$ as $P\to +\infty$. Since as already noticed  
 $f''(P)>0$, then  $f'(P)$ is monotonically increasing and it has an  unique zero $P_0$, which in additions verifies $P^+<P_0<P^-$.
%
 In the interval $P^+\leq P<P^-$ we have 
 \begin{equation*}
 F(P)-F(P^-)= \frac{1}{k^2} \int_{P^-}^P f(z) dz>0.
 \end{equation*}
 Moreover $F'(P)=f(P)>0$ for $0<P<P^+$ and $F(P) \to  - \infty$ as $P \to 0^+$.  Hence,  there is a point $P^*\in (0,P^+)$ such that $F(P^*)-F(P^-)=0$. Therefore,   $H(P^*,0)=0$,
 $H(P,0)<0$ for $0<P<P^*$ and $H(P,0)>0$ for $P^*<P<P^+$.
 Then,  reduced system   \eqref{first_order_system_reduced_1}-\eqref{first_order_system_reduced_2} has an  homoclinic loop, starting at $[P^-,0]$ and passing through the point $[P^*,0]$, which is contained in the  level set $H(P,Q)=0$.  If we express $Q$ as a function of $P$ from $H(P,Q)=0$, the homoclinic loop can be expressed by the two branches
\begin{equation*}
Q=\pm \sqrt{2(F(P)-F(P^-))},
\end{equation*}
for $P^* \leq P \leq P^-$; see Figure \ref{fig_heteroclinic2}.

We want now to prove that the homoclinic loop of \eqref{first_order_system_reduced_1}-\eqref{first_order_system_reduced_2} defines a confining set for 
 \eqref{first_order_system_1}-\eqref{first_order_system_2}. Indeed, first of all we see that $H(P,0)= F(P)-F(P^-) > 0$ for any $P\in (P^*,P^-)$ and therefore $H(P,Q) >0$ in the interior of the homoclinic loop. Moreover, 
let us  consider  a trajectory $[P(y),Q(y)]$  solution of \eqref{first_order_system_1}-\eqref{first_order_system_2} and let us define $\mathcal{H}(y):=H(P(y),Q(y))$. We have
\begin{equation*}
\mathcal{H}' = \frac{\partial H}{\partial P}P'+\frac{\partial H}{\partial Q}Q'=\frac{2 s \mu}{k^2}Q^2 \geq 0.
\end{equation*}
Since $\mathcal{H}'\geq 0$ for all points of the homoclinic loop, we conclude that  a trajectory starting inside it at $y=0$ will stay inside for all $y\geq 0$.
\begin{figure}
\begin{center}
\includegraphics[scale=0.8]{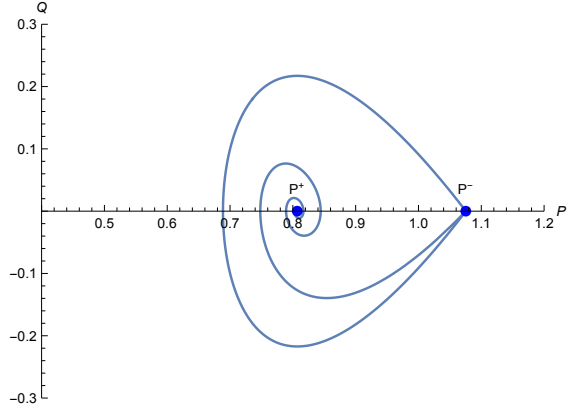}
\end{center}
\caption{The homoclimic loop and the heteroclinic connection for parameters $A=1$, $B=-3.1$, $s=1$, $\gamma=3/2$, $\mu=0.3$, $k=1$}
\label{fig_heteroclinic2}
\end{figure}

Now we are going to show, that the eigenvector, tangent to the unstable subspace of the steady-state $[P^-,0]$ is pointing inside the homoclinic loop. The linearization of \eqref{first_order_system_1}-\eqref{first_order_system_2} at $P^{\pm}$ and $Q=0$ is $J$ and the linearization of \eqref{first_order_system_reduced_1}-\eqref{first_order_system_reduced_2} at $P^{\pm}$ and $Q=0$ is $\tilde{J}$, where
\begin{equation*}
J=\begin{bmatrix}
0 & 1\\
\frac{f'(P^{\pm})}{k^2} & -\frac{2 s \mu}{k^2}
\end{bmatrix},\mbox{ }
\tilde{J}=\begin{bmatrix}
0 & 1\\
\frac{f'(P^{\pm})}{k^2} & 0
\end{bmatrix}.
\end{equation*}
The eigenvalues of $J$ are
\begin{equation*}
\lambda_{1,2}=\frac{-s \mu \pm \sqrt{k^2 f'(P^\pm)+s^2 \mu^2}}{k^2},
\end{equation*}
while the eigenvalues of $\tilde{J}$ are
\begin{equation*}
\tilde{\lambda}_{1,2}=\pm\frac{1}{k}\sqrt{f'(P^\pm)}.
\end{equation*}
At  the steady-state $[P^-,0]$, since $f'(P^-)>0$, we have $\lambda_1>0$ and $\lambda_2<0$ and therefore  $[P^-,0]$ is a saddle for  \eqref{first_order_system_1}-\eqref{first_order_system_2}. The eigenvector of $J$ corresponding to $\lambda_1$,  which  is tangent to the unstable manifold of the saddle,
is given by
\begin{equation*}
v_1=-\begin{bmatrix}
\dfrac{s \mu + \sqrt{k^2 f'(P^-) + s^2 \mu^2}}{f'(P^-)} \\ 1
\end{bmatrix}.
\end{equation*}

Now consider the linearization of \eqref{first_order_system_reduced_1}-\eqref{first_order_system_reduced_2} at $[P^-,0]$. For the eigenvalues we have $\tilde{\lambda}_1>0$ and $\tilde{\lambda}_2<0$. The eigenvector of $\tilde{J}$, corresponding to the unstable eigenvalue $\tilde{\lambda}_1$ is
\begin{equation*}
\tilde{v}_1=-\begin{bmatrix}
\dfrac{k}{\sqrt{f'(P^-)}} \\ 1
\end{bmatrix}.
\end{equation*}
If $\tilde{v}_{1,1}>v_{1,1}$, then the eigenvector points inside the homoclinic loop. Since
\begin{equation*}
2 s^2 \mu^2 + 2 s \mu \sqrt{k^2 f'(P^-)+s^2 \mu^2}>0,
\end{equation*}
we get
\begin{equation}
(s \mu + \sqrt{k^2 f'(P^-)+s^2 \mu^2})^2 > k^2 f'(P^-).\label{vector_direction}
\end{equation}
The inequality $\tilde{v}_{1,1}>v_{1,1}$ follows from \eqref{vector_direction} by taking a square root and dividing by $f'(P^-)$.
As a consequence, being the set defined by  the homoclinic loop a confining set   for 
 \eqref{first_order_system_1}-\eqref{first_order_system_2}, we conclude that the orbit exiting from  $[P^-,0]$  tangent to the eigenvector $v_1$ is trapped inside it.

Now, let us examine the linearization 
at the steady-state $[P^+,0]$. Since we have $f'(P^+)<0$, then either 
$$k^2 f(P^+)+s^2 \mu^2<0,$$ 
or 
$$k^2 f(P^+)+s^2 \mu^2 \geq 0\ \hbox{and}\ \sqrt{k^2 f(P^+)+s^2 \mu^2}<s \mu.$$ 
In both cases,  $\Re \lambda_{1}, \Re \lambda_{2}<0$, that is the steady-state $[P^+,0]$ is stable. 

At this point, we need first to exclude that orbits 
inside the homoclinic loop  converge to $[P^-,0]$.
To this end, the eigenvector of $J$ at this point,  corresponding to the stable eigenvalue $\lambda_2$,
is given by
\begin{equation*}
v_2=\begin{bmatrix}
\dfrac{-s \mu + \sqrt{k^2 f'(P^-) + s^2 \mu^2}}{f'(P^-)} \\ -1
\end{bmatrix},
\end{equation*}
while the eigenvector of $\tilde{J}$, again at $[P^-,0]$ and corresponding to $\tilde{\lambda}_2$, is
\begin{equation*}
\tilde{v}_2=\begin{bmatrix}
\dfrac{k}{\sqrt{f'(P^-)}} \\ -1
\end{bmatrix}.
\end{equation*}
If $\tilde{v}_{2,1}>v_{2,1}$, then the eigenvector, tangent to the stable manifold of the saddle is pointing towards $[P^-,0]$ from outside the homoclinic loop. 
Since 
$$(k \sqrt{f'(P^-)}+s \mu)^2>k^2 f'(P^-)+s^2 \mu^2>0,$$
 taking a square root we get 
 $$k \sqrt{f'(P^-)}>-s \mu+\sqrt{k^2 f'(P^-)+s^2 \mu^2}$$
  and dividing by $f'(P^-)$ we obtain $\tilde{v}_{2,1}>v_{2,1}$. As a consequence,  no orbits inside the homoclinic loop can converge to $[P^-,0]$.

Now, let $\Omega$ be the compact set defined by the homoclinic loop and its interior. Clearly, as proved before,  $\Omega$ is an invariant set for 
 \eqref{first_order_system_1}-\eqref{first_order_system_2}. In $\Omega$, consider the function
\begin{equation*}
V(P,Q)=\frac{Q^2}{2}-F(P)+F(P^+)=F(P^+)-F(P^-)-H(P,Q),
\end{equation*}
and denote by $\mathcal{V}(y): =V(P(y),Q(y))$ the dynamics of this function along an orbit of the dynamical system  \eqref{first_order_system_1}-\eqref{first_order_system_2}. 
We have $V(P^+,0)=0$ and $\mathcal{V}'=-\mathcal{H}' \leq 0$ in $\Omega$, that is  $V$ is a Lyapunov function for  \eqref{first_order_system_1}-\eqref{first_order_system_2}.
Denote by $E$  the set of points in $\Omega$ where $\mathcal{V}'=0$, that is $E=\Omega\cap\{Q=0\}$. Hence,  the LaSalle invariance principle implies that any solution $[P(y),Q(y)]$ in $\Omega$  will converge to the largest invariant subset of $E$, which is the set of two steady-states $[P^+,0]$ and $[P^-,0]$. Finally, since as proven before $[P^-,0]$ can not be reached by orbits inside $\Omega$, we can conclude that the trajectory exiting along the unstable manifold of $[P^-,0]$ will converge to the stable  steady-state $[P^+,0]$, proving the existence of the desired heteroclinic connection.
Finally, if $k^2 f(P^+)+s^2 \mu^2<0$,  the eigenvalues of the linearization at this point have nonzero imaginary parts, which proves that  in this case the heteroclinic will be oscillatory in a neighbourhood of $[P^+,0]$.

\emph{Case (ii).}
To prove the result in the case  $s<0$ and $P^-<P^+$, we shall  reverse the parameter  describing the profile, namely, we shall consider $\tilde{y}=-y$ and the dependent variables $\tilde{P}$ and $\tilde{Q}$ accordingly. 
Denoting $'=d/d\tilde{y}$,  the dynamical system rewrites as follows:
\begin{align}
\tilde{P}'&=\tilde{Q},\label{first_order_system_reverse_1}\\
\tilde{Q}'&=\frac{f(\tilde{P})}{k^2}-\frac{2 \tilde{s} \mu}{k^2}\tilde{Q},\label{first_order_system_reverse_2}
\end{align}
where $\tilde{s}=-s$, $\tilde{s}>0$.
Hence, we can apply Case (i) to  \eqref{first_order_system_reverse_1}-\eqref{first_order_system_reverse_2} 
to conclude there exists an heteroclinic orbit for that system connecting $[P^+,0]$ to $[P^-,0]$. The latter  corresponds to an heteroclinic orbit connecting $[P^-,0]$ to $[P^+,0]$ in the forward parameter $y$ for \eqref{first_order_system_1}-\eqref{first_order_system_2} for $s<0$ and $P^-<P^+$ and the proof is complete.
\end{proof}

We conclude the section by stating the existence result proved above in terms of the speed $s$ and the end states $[\rho^{\pm}$, $u^{\pm}]$ of the system \eqref{euler1}-\eqref{euler2}, linked by the Rankine-Hugoniot conditions \eqref{RHnon1}-\eqref{RHnon2}, and in particular in terms of   Lax entropy conditions, as well as super- or subsonic conditions. 
To this end, let us denote $W=[\rho,u]$. The eigenvalues of the Jacobian of the system \eqref{euler1}-\eqref{euler2} are
\begin{equation*}
\lambda_1(W) = u - c_s(\rho),\mbox{ }\lambda_2(W) = u + c_s(\rho),
\end{equation*}
where, noting that $h'(\rho)\geq 0$, we used the notation $c_s(\rho) = \sqrt{\rho h'(\rho)}\geq0$ for the sound speed. We recall that a discontinuity $(W^\pm,s)$ verifyng  the Rankine-Hugoniot conditions \eqref{RHnon1}-\eqref{RHnon2} is a Lax $k$--shock, $k=1,2$, if
\begin{equation*}
\lambda_k(W^+)<s<\lambda_k(W^-).
\end{equation*}
Moreover, the state $W^\pm = [\rho^\pm,u^\pm]$ is referred to as supersonic (resp.\ subsonic) if $|u^\pm| > c_s(\rho^\pm)$ (resp.\ $|u^\pm| < c_s(\rho^\pm)$).
\begin{corollary}
\label{global_existence_entropy}
Suppose the end states $[\rho^{\pm}$, $u^{\pm}]$ and the speed $s$ satisfy $\rho^{\pm}>0$ and $[\rho^{\pm}$, $u^{\pm}; s]$ defines
\begin{enumerate}
\item[\textit{(i)}]
 a Lax 2--shock with a subsonic right state;
\item[\textit{(ii)}]
  a Lax 1--shock with a subsonic left state. 
\end{enumerate}
Then there exists a traveling wave profile connecting 
$[\rho^{-}$, $u^{-}]$ to $[\rho^{+}$, $u^{+}]$.
\end{corollary}
\begin{proof}
First of all, we are going to express $u^{\pm}$ in terms of $\rho^{\pm}$ from the Rankine-Hugoniot conditions \eqref{RHnon1}-\eqref{RHnon2}. From equation \eqref{RHnon1} we get
\begin{equation}
u^+ = s + \frac{\rho^-}{\rho^+}(u^--s) \label{eq_up}.
\end{equation}
Substituting $u^+$ in equation \eqref{RHnon2} and dividing by 
\begin{equation*}
 \frac{1}{2}\Big{(}1-\Big{(}\frac{\rho^-}{\rho^+}\Big{)}^2\Big{)} \neq 0,
 \end{equation*}
 we obtain the quadratic equation
 \begin{equation}
 (u^-)^2 - 2s u^{-} + s^2 + 2 (\rho^+)^2\frac{h(\rho^-)-h(\rho^+)}{(\rho^+)^2-(\rho^-)^2}=0.
 \label{eq_quad}
 \end{equation}
 Since $h(\rho)$ is strictly increasing in $\rho>0$,  the discriminant  of this equation 
 \begin{equation*}
  8 (\rho^+)^2 \frac{h(\rho^+)-h(\rho^-)}{(\rho^+)^2-(\rho^-)^2} = 
  8 \frac{(\rho^+)^2}{\rho^+ + \rho^-} \frac{h(\rho^+)-h(\rho^-)}{(\rho^+) - (\rho^-)}
 \end{equation*}
 is strictly positive.
 Hence, the two real  roots of \eqref{eq_quad} are given by $u^{-}_{1,2}=s \pm d$,
 where 
 \begin{equation*}
 d =  \rho^{+} \sqrt{2\frac{h(\rho^+)-h(\rho^-)}{(\rho^+)^2-(\rho^-)^2}}>0.
 \end{equation*}
 Substituting these roots in equation \eqref{eq_up} yields the two solutions
 \begin{equation*}
  u^+_{1,2} = s \pm \frac{\rho^-}{\rho^+} d.
 \end{equation*}
 
 \emph{Case (i).}
 Since   the shock satisfies the Lax condition
 \begin{equation*}
 \lambda_2(W^+)<s<\lambda_2(W^-)
 \end{equation*}
 and $W^+$ is subsonic, in particular we have $s>u^+ +c_s(\rho^+)>0$. 
 Moreover, since $c_s(\rho^+)>0$, we conclude   
 $$u^+=s - \frac{\rho^-}{\rho^+} d$$
 and, accordingly, $u^- =  s-d$.
 Using again   the Lax condition we get 
 $$u^+ - u^- <c_s(\rho^-) - c_s(\rho^+),$$
  that is  
 \begin{equation*}
  ( {\rho^+ - \rho^-})\frac{d}{\rho^+} < c_s(\rho^-) - c_s(\rho^+).
 \end{equation*}
 Since the speed of sound $c_s(\rho)$ is non decreasing, from the above inequality we conclude $\rho^+<\rho^-$ and we are in Case (i)  of Lemma \ref{global_existence_nonlinear_dissipation} for the existence of a profile.

 \emph{Case (ii).}
In this case, the shock satisfies the Lax condition
 \begin{equation*}
 \lambda_1(W^+)<s<\lambda_1(W^-)
 \end{equation*}
and, being   $W^-$   subsonic, we conclude $s<u^- - c_s(\rho^-)<0$. 
Moreover, since $c_s(\rho^-)>0$, we have $u^{-}=s+d$ and 
 $$u^+=s + \frac{\rho^-}{\rho^+} d.$$ 
In addition, from the Lax condition 
we infer
 $$u^+-u^-<c_s(\rho^+)-c_s(\rho^-),$$ 
 which implies 
 \begin{equation*}
 (\rho^- - \rho^+)\frac{d}{\rho^+}<c_s(\rho^+)-c_s(\rho^-).
 \end{equation*}
 As before, this inequality implies $\rho^+>\rho^-$ because the  sound speed $c_s(\rho)$ is non-decreasing.
Finally, we are in Case (ii)  of Lemma \ref{global_existence_nonlinear_dissipation} and we can conclude with  the existence of a profile.
\end{proof}
\begin{remark}\label{rem:sonic}
It is worth observing that the conditions stated in Corollary \ref{global_existence_entropy} are only possible \emph{sufficient} conditions which guarantee the validity of the hypotheses of Lemma \ref{global_existence_nonlinear_dissipation}, while other possible regimes for the end states may be considered as well. More precisely, in both cases the subsonic assumptions on the end states are needed solely to determine the sign of the speed of the traveling wave $s$, which can be clearly obtained in many other cases. For instance, one could replace Case (i) with the case of a Lax 2--shock with a right state with positive velocity  (to have 
$s>u^+ +c_s(\rho^+)>0$), or  replace Case (ii) with the case of a Lax 1--shock with a left state with negative velocity  (to have 
$s<u^- -c_s(\rho^-)<0$). However, we decided to present the results as in Corollary \ref{global_existence_entropy} to have a concise statement which refers to well-known physical properties of the end states $[\rho^\pm, u^\pm]$. 
\end{remark}
\section{numerical sensitivity for oscillations in terms of nearness to vacuum and ratio viscosity/dispersion}\label{sec:sens}
The aim of this section is to give numerically some insights on the behavior of the profiles whose existence has been proved  in Section \ref{sec:existence}. As it is manifest (see in particular Corollary \ref{global_existence_entropy}), besides Lax conditions and subsonic conditions for the end sates, in the nonlinear viscosity case under investigation here, there are no other requirements to  prove  this results, in contrast to what happens in the linear case \cite{Zhelyazov}. Indeed, in this paper, the analytical proof of existence of profiles for large amplitude  shocks and linear viscosity requires that the ratio $\mu/k$ is sufficiently big, in addition to  conditions on end states similar to the ones presented here. Moreover, conditions analogous to those of Lemma \ref{global_existence_nonlinear_dissipation} guarantee the profile is non monotone also for the linear viscosity case. This behavior is validated numerically in  \cite{LMZ2020}, where in addition it is shown numerically that the existence of profiles appears to be valid beyond the sufficient conditions needed for the analytical proof.

Consistently, in the nonlinear model under examination in this paper, the role of the ratio $\mu/k$, in connection to the end states $P^{\pm}$,  is confined only in determining the monotonicity property of the profiles,  while their existence  is independent on it: as is expected, the presence of larger dispersion  increases the oscillations of the profiles, and we shall give numerical evidence of this behavior in the subsequent sections.
For this, we will use the variable $R(y)=P(y)^2$, denoting the density profile, while  $U(y)$ is deduced  from equation \eqref{eq_v}. Moreover, once we give end states $\rho^{\pm} = {(P^2)}^{\pm}$, the corresponding values for the velocity $u^\pm$ are given by
%
%
%
\begin{align}
  u^+_{1,2} &= s \pm \frac{\rho^-}{\rho^+} d, \label{up12} \\
 u^{-}_{1,2} &=s \pm d, \label{um12}
 \end{align}
 where
 \begin{equation*}
 d =  \rho^{+} \sqrt{2\frac{h(\rho^+)-h(\rho^-)}{(\rho^+)^2-(\rho^-)^2}}>0.
 \end{equation*}
 For the sake of completeness and clarity, 
let us now analyze the aforementioned two possibilities for the velocities $u^\pm$ in terms of the Lax conditions, starting for the case $0<\rho^+ < \rho^-$. More precisely, in that case, we shall prove that
the velocity $u^{\pm}_1$ should not be considered, because the corresponding shock $(\rho^{\pm},u^{\pm}_1,s)$ will not be an admissible Lax shock for \eqref{euler1}-\eqref{euler2}, while $u^{\pm}_2$ will define a Lax 2--shock for that system. 

To this end, let us suppose $0<\rho^+<\rho^-$ and assume $u^{\pm}_1$ defines a Lax 1--shock, that is
\begin{equation*}
u^{+}_1 - c_s(\rho^+) < s < u^{-}_1 - c_s(\rho^-).
\end{equation*}
This implies
\begin{equation*}
\frac{\rho^-}{\rho^+}d - c_s(\rho^+) < 0 < d - c_s(\rho^-),
\end{equation*}
which is impossible because   
$$\frac{\rho^-}{\rho^+}d>d$$
 and, since $c_s(\rho)$ is non-decreasing,  $-c_s(\rho^+) \geq -c_s(\rho^-)$. 
If $u^{\pm}_1$   satisfies the condition for Lax 2--shock, then  in particular $u^{+}_1 + c_s(\rho^+) < s$. Using  again the expression \eqref{up12} for $u^+_1$ we end up with 
\begin{equation*}
\frac{\rho^-}{\rho^+} d + c_s(\rho^+) < 0,
\end{equation*}
which is impossible because  $\rho^\pm$, $d$, and  $c_s(\rho^+)$ are strictly positive. Hence, the solution $u^{\pm}_1$ can not be considered because it defines a discontinuity which is not admissible.

Now, let us move to check  the conditions verified by $u^{\pm}_2$, starting by proving it does not satisfy those of Lax 1--shock. 
Indeed, if we assume by contradiction these conditions are verified, then in particular we have $s < u^{-}_2 - c_s(\rho^-)$. 
Using the expression \eqref{um12} for $u^-_2$, this is equivalent to
\begin{equation*}
0 > d + c_s(\rho^-),
\end{equation*}
which is again impossible because $d$, $c_s(\rho^-)>0$. Finally, let us check $u^{\pm}_2$ verifies  the conditions for a Lax 2--shock, namely
\begin{equation}\label{eq:sec4Lax}
u^{+}_2 + c_s(\rho^+) < s < u^{-}_2 + c_s(\rho^-).
\end{equation}
In the following we will use the ratio $r = \frac{\rho^+}{\rho^-}$, which verifies $0 < r < 1$, being  $0<\rho^+ < \rho^-$, and we shall consider the two cases $\gamma = 1$ and $\gamma >1$ separately.

Let us  consider the case $\gamma = 1$. From \eqref{up12} for $u^+_2$,  the relation  $u^{+}_2 + c_s(\rho^+) <~s$
is equivalent to
\begin{equation*}
1 < \rho^{-} \sqrt{2\frac{\ln(\rho^+)-\ln(\rho^-)}{(\rho^+)^2-(\rho^-)^2}}.
\end{equation*}
Squaring and multiplying by $r^2 - 1<0$, we get the equivalent inequality
\begin{equation}
r^2 - 1 > 2 \ln r, \mbox{ for } 0 < r=  \frac{\rho^+}{\rho^-} < 1. \label{Lax_2s_1}
\end{equation}
With the notations
$
f(r) = r^2 - 1$, $g(r) = 2 \ln r$,
we have $f(1) =g(1) = 0$ and
\begin{equation*}
f(r) - g(r) = \int_r^1 (g'(z) - f'(z)) dz =  \int_r^1 2 \left (\frac{1}{z} - z \right ) dz >0,
\end{equation*}
which proves \eqref{Lax_2s_1}.
The second inequality in \eqref{eq:sec4Lax} is proved similarly. Indeed, 
using the expression \eqref{um12} for $u^-_2$, $ s < u^{-}_2 + c_s(\rho^-)$ rewrites
\begin{equation*}
\rho^{+} \sqrt{2\frac{\ln(\rho^+)-\ln(\rho^-)}{(\rho^+)^2-(\rho^-)^2}} < 1.
\end{equation*}
Squaring and multiplying by $1 - r^{-2}<0$ we end up to
\begin{equation}
2 \ln r > 1 - \frac{1}{r^2}, \mbox{ for } 0 < r =  \frac{\rho^+}{\rho^-} < 1. \label{Lax_2s_2}
\end{equation}
Considering this time
\begin{equation*}
g(r) = 2 \ln r ,\mbox{ } \tilde f(r) = 1 - \frac{1}{r^2},
\end{equation*}
one has $\tilde f(1) = g(1) = 0$ and
\begin{equation*}
g(r)- \tilde f(r)  = \int_r^1(\tilde{f}'(z) - g'(z))dz =  \int_r^1 2  \left ( \frac{1}{z^3} - \frac{1}{z}  \right ) dz >0.
\end{equation*}
Hence \eqref{Lax_2s_2} is satisfied   and we have checked the Lax condition \eqref{eq:sec4Lax} for $\gamma=1$.

Now, let us consider the case $\gamma > 1$. 
Substituting the value $u^+_2$ given in \eqref{up12} in the inequality $u^{+}_2 + c_s(\rho^+) < s$, one has
\begin{equation*}
\sqrt{\gamma(\rho^+)^{\gamma-1}}<  \rho^{-} \sqrt{\frac{2\gamma}{\gamma-1}\frac{(\rho^+)^{\gamma-1}-(\rho^-)^{\gamma-1}}{(\rho^+)^2-(\rho^-)^2}}.
\end{equation*}

Squaring and multiplying by
\begin{equation*}
\frac{1}{\gamma (\rho^+)^{\gamma - 1}} \Big{(} \Big{(} \frac{\rho^+}{\rho^-} \Big{)}^2 - 1 \Big{)} < 0, 
\end{equation*}
we obtain 
\begin{equation}
r^2 - 1 > \frac{2}{\gamma - 1}(1 - r^{1-\gamma}), \mbox{ for } 0 < r =  \frac{\rho^+}{\rho^-}< 1 \label{Lax_2s_3}.
\end{equation}
For
\begin{equation*}
f(r) = r^2 - 1 \mbox{ and } \tilde g(r) = \frac{2}{\gamma - 1}(1 - r^{1 - \gamma}),
\end{equation*}
we have $f(1) = \tilde g(1) = 0$. Moreover
\begin{equation*}
f(r) - \tilde g(r) = \int_r^1( \tilde{g}'(z) - f'(z))dz > 0,
\end{equation*}
because $2/ r^\gamma > 2 r$, for $0 < r < 1$ and $\gamma>1$, and   \eqref{Lax_2s_3} is verified. 
%
Now, using   the expression \eqref{um12} for $u^-_2$, we rewrite the second inequality in \eqref{eq:sec4Lax} as follows:
\begin{equation*}
\rho^{+} \sqrt{\frac{2\gamma}{\gamma-1}\frac{(\rho^+)^{\gamma-1}-(\rho^-)^{\gamma-1}}{(\rho^+)^2-(\rho^-)^2}} < 
\sqrt{\gamma(\rho^-)^{\gamma-1}}.
\end{equation*}
Squaring and multiplying by
\begin{equation*}
\frac{1}{\gamma (\rho^-)^{\gamma - 1}}\Big{(} 1 - \Big{(} \frac{\rho^-}{\rho^+}\Big{)}^2\Big{)} < 0,
\end{equation*}
the above inequality    is equivalent to
\begin{equation}
\frac{2}{\gamma - 1}(r^{\gamma-1} - 1) > 1 - \frac{1}{r^2}, \mbox{ for } 0 < r = \frac{\rho^+}{\rho^-} < 1. \label{Lax_2s_4}
\end{equation}
Let us now consider
\begin{equation*}
\bar g (r) = \frac{2}{\gamma - 1}(r^{\gamma - 1} - 1),\mbox{ }\tilde f(r) = 1 - \frac{1}{r^2}.
\end{equation*}
Then we have $\bar g (1) = \tilde f(1) = 0$ and 
\begin{equation*}
\bar g(r) - \tilde f(r) = \int_{r}^1(\tilde{f}'(z) - \bar{g}'(z))dz > 0, 
\end{equation*}
being  $2/r^3 > 2/r^{2-\gamma}$ for $0 < r < 1$ and $\gamma>1$. Hence,   \eqref{Lax_2s_4} is satisfied and  the Lax condition \eqref{eq:sec4Lax}  is verified  also for $\gamma>1$.

In the case $0<\rho^- < \rho^+$, similar arguments show that the velocity $u^{\pm}_2$ should not be considered, because the resulting shock $(\rho^{\pm},u^{\pm}_2,s)$ would not be admissible for \eqref{euler1}-\eqref{euler2}, while $u^{\pm}_1$ leads to an admissible Lax 1--shock; we leave the details to the reader.

\subsection{Sensitivity for ratio of viscosity to dispersion coefficients}\label{subsec:sensvisc}
We start by analyzing numerically how the  profile is affected by the ratio $\mu/k$. For this,  we consider the model \eqref{eq_sys_n_intro} with adiabatic exponent  $\gamma=5/3$ and dispersion coefficient $k=\sqrt{2}$, leaving $\mu$ and hence  $\mu/k$ variable. In particular, we consider four values for this ratio, that is, in decreasing order: 
 2.83, 0.71, 0.35, and  0.18.
Moreover, the end states for the density are  $\rho^+=1.0 < \rho^-=1.5$, the speed of the profile is $s=1$, and 
 the corresponding values for the velocity defining an admissible Lax 2--shock, as said in the discussions above, are $u^+=-0.67 $ and  $u^-=-0.11$.
In particular, the sufficient conditions for existence of the profile stated in Lemma \ref{global_existence_nonlinear_dissipation}, case (i), are verified, as well a
  conditions (i) of Corollary \ref{global_existence_entropy}, being in addition
\begin{equation*}
|u^+| = 0.67 < 1.29 = c_s(\rho^+).
\end{equation*}

The behavior of the resulting profile in the four cases considered above is depicted in  
 Figure \ref{fig_sensitivity_viscosity}.
 The first profile is monotone, in agreement to the analytical result of Lemma \ref{global_existence_nonlinear_dissipation}, case (i), being  
\begin{equation*}
\frac{s\mu}{k} = 2.83 > 1.50 =  \sqrt{-f'(P^+)},
\end{equation*}
and the Jacobian $J$ of the dynamical system \eqref{first_order_system_1}-\eqref{first_order_system_2} at $[P^+,0]$ has real eigenvalues.
Then, since 
\begin{equation*}
0.18< 0.35 < 0.71< 1.50 =  \sqrt{-f'(P^+)},
\end{equation*}
 in all remaining cases the condition is verified, and the profile turns out to be  oscillatory, as predicted by the lemma. Moreover, the numerics shows also that the oscillations increases as the ratio $\mu/k$ decreases.
 In these last three cases, the Jacobian $J$  at $[P^+,0]$ has imaginary eigenvalues, and their 
  imaginary parts  increase in modulus as the ratio $\mu/k$ decreases.
\begin{figure}[H]
\centering
\begin{subfigure}[b]{0.45\textwidth}
\includegraphics[scale=0.6]{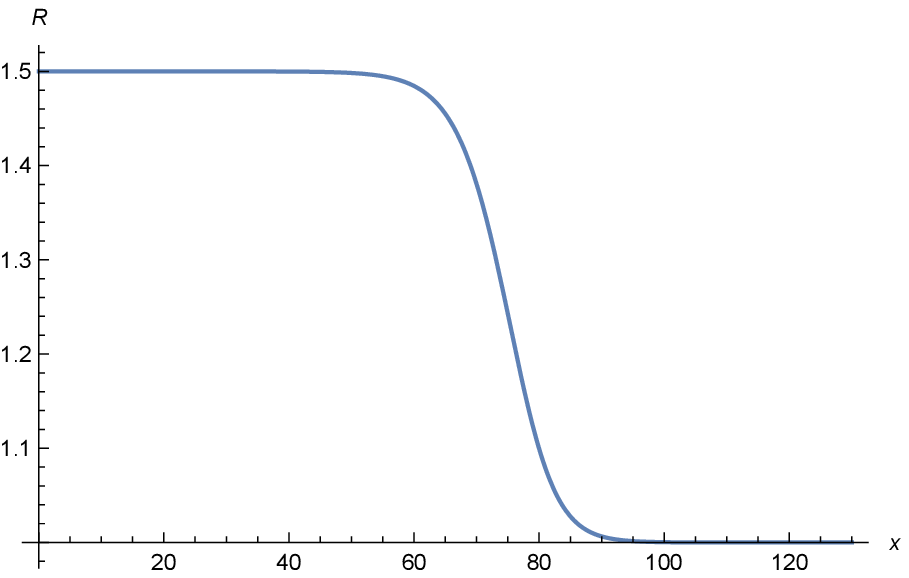}
\caption{}
\end{subfigure}%
\begin{subfigure}[b]{0.45\textwidth}
\includegraphics[scale=0.6]{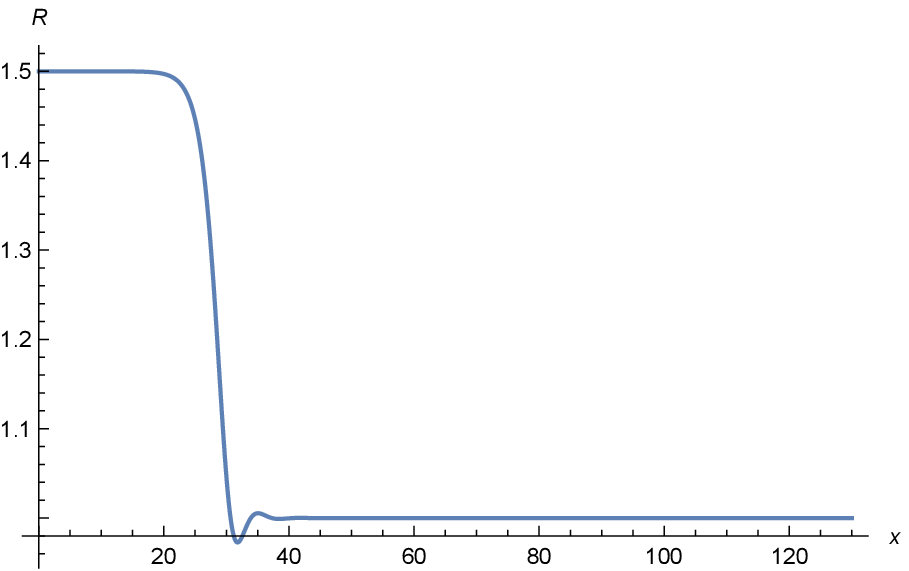}
\caption{}
\end{subfigure}

\centering
\begin{subfigure}[b]{0.45\textwidth}
\includegraphics[scale=0.6]{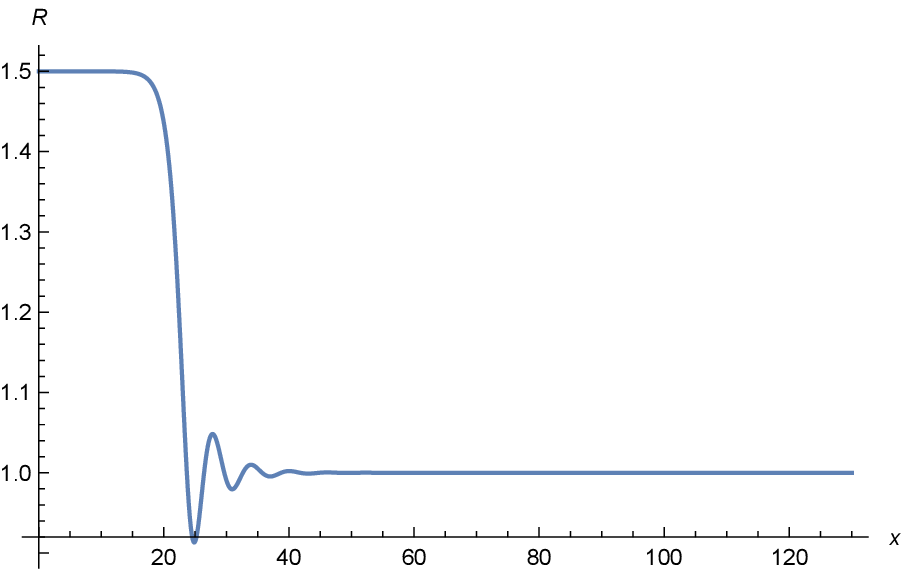}
\caption{}
\end{subfigure}%
\begin{subfigure}[b]{0.45\textwidth}
\includegraphics[scale=0.6]{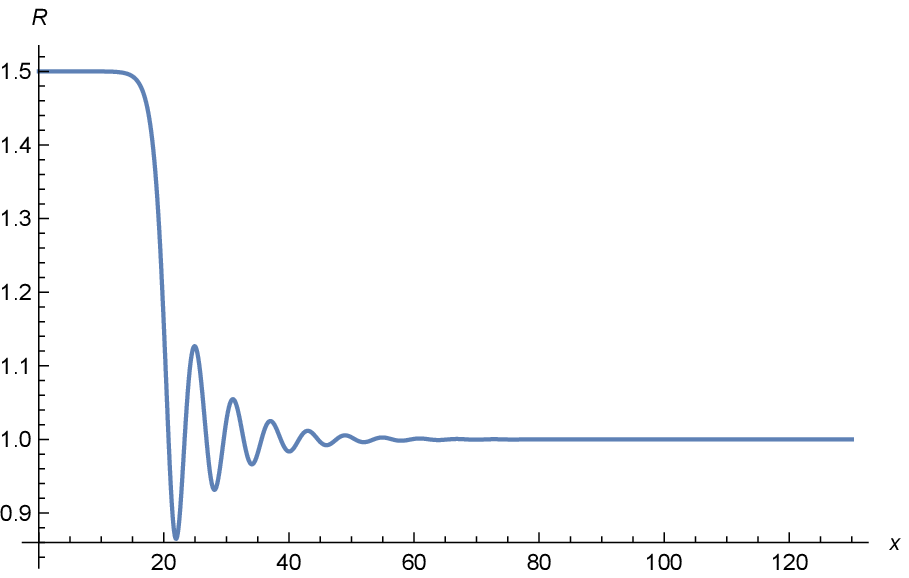}
\caption{}
\end{subfigure}
\caption{Profiles for different values of the viscosity coefficient $\mu$. The parameters are $s=1$, $\gamma=5/3$, $k=\sqrt{2}$, $\rho^+=1.0$, $\rho^-=1.5$, $u^+=-0.67$, $u^-=-0.11$.  (A) $\mu/k=2.83$. (B) $\mu/k=0.71$. (C) $\mu/k=0.35$. (D) $\mu/k=0.18$.}
\label{fig_sensitivity_viscosity}
\end{figure}
\subsection{Sensitivity analysis about vacuum}\label{subsec:vacuum}
Our second numerical experiments aim to study  the behavior of the profiles as the end state $\rho^+$ varies and all other parameters are fixed. For this, we consider the system \eqref{eq_sys_n_intro} with $\gamma = 3/2$, $\mu=1.2$ and $k=\sqrt{2}$. The speed of the profile $s$ is kept fixed equal to 1, as well as the left state $\rho^-=0.5$, while we consider the following  four values for the right state $\rho_{a,b,c,d}^+$: 0.4, 0.3, 0.1, 0.05. 
The corresponding values for the velocity defining  an admissible  Lax 2--shock are given by
\begin{align*}
& u^-_a = 0.11, & & u^+_a = -0.12,\\
& u^-_b = 0.27, & & u^+_b = -0.22,\\
& u^-_c = 0.69, & & u^+_c = -0.56,\\
& u^-_d = 0.83, & & u^+_d = -0.71.\\
\end{align*}
Hence, Lemma \ref{global_existence_nonlinear_dissipation}, case (i),  applies to these sets of parameters, and, in addition,  as said before, the resulting shocks 
$((\rho^-,u^-_{a,b,c,d}), (\rho^+_{a,b,c,d},u^+_{a,b,c,d}),s)$ are   admissible Lax 2--shocks. Since  the values of $|u^+_{a,b,c,d}|$  are 0.12, 0.22, 0.56, 0.71, and those of $c_s(\rho^+_{a,b,c,d})$ are 0.97, 0.91, 0.69, 0.58,  the first three parameter sets satisfy also the condition for a subsonic right state   of Corollary \ref{global_existence_entropy}, case (i):
\begin{equation*}
0.12 < 0.97,\mbox{ }0.22 < 0.91,\mbox{ }0.56 < 0.69.
\end{equation*}
Conversely, for $\rho_d^+ = 0.05$, the condition is violated because $0.71 > 0.58$, and  Corollary \ref{global_existence_entropy} does not apply in this case.

 The profiles are shown in Figure \ref{figure_profiles_vacuum}. The values of $\sqrt{-f'(P_{a,b,c,d}^+)}$, corresponding to $\rho_{a,b,c,d}^+$, are 0.77, 1.16, 1.98, 2.28. The profile with $\rho_a^+=0.4$ is monotone, in accordance with Lemma \ref{global_existence_nonlinear_dissipation}, case (i), because 
\begin{equation*}
\frac{s\mu}{k} = 0.85 > 0.77 =  \sqrt{-f'(P_a^+)},
\end{equation*}
and the Jacobian $J$ at $[P_a^+,0]$ has real eigenvalues. Moreover,
\begin{equation*}
\frac{s\mu}{k}= 0.85 < 1.16 < 1.98 < 2.28,
\end{equation*}
so in all other cases the condition of  Lemma \ref{global_existence_nonlinear_dissipation}, case (i), is verified and the profiles are non-monotone.
As we have pointed out also in the section before, we see that the profiles become more oscillatory as the end state $\rho^+$ becomes closer to vacuum; however, the oscillations have smaller amplitude compared to the case where we vary the viscosity. Finally, for the last three parameter sets, namely at 
$[P_{b,c,d}^+,0]$,  the Jacobian $J$   has imaginary eigenvalues, and their imaginary parts  increase in absolute value as the end state $\rho^+$ approaches vacuum.
\begin{figure}[H]
\centering
\includegraphics[scale=0.6]{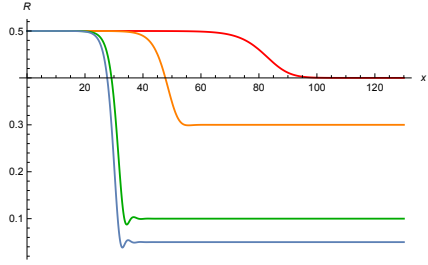}

\caption{Profiles for varying end state $\rho^+$. The parameters are $s=1$, $\gamma=3/2$, $k=\sqrt{2}$, $\rho^-=0.5$, $\mu=1.2$. red: $\rho_a^+=0.4$, $u^-_a=0.11$, $u^+_a=-0.12$;  orange: $\rho_b^+=0.3$,  $u^-_b=0.27$, $u^+_b=-0.22$;  green: $\rho_c^+=0.1$, $u^-_c=0.69$, $u^+_c=-0.56$;  blue: $\rho_d^+=0.05$, $u^-_d=0.83$, $u^+_d=-0.71$.}
\label{figure_profiles_vacuum}
\end{figure}

\end{document}